\documentclass[a4paper, reqno]{amsart} 

\usepackage{enumerate}

\usepackage{indentfirst} 
\usepackage{newlfont}   

\usepackage{amsthm, amssymb, amsmath, latexsym}
\usepackage{graphicx}

\usepackage{amsfonts, amscd, amsbsy, mathrsfs}

\theoremstyle{plain}                     
\begingroup
\newtheorem{teo}{Theorem}[section]       
\newtheorem{prop}[teo]{Proposition}     
        
\newtheorem{lem}[teo]{Lemma}      
\endgroup
      
\theoremstyle{definition}                
\begingroup
\newtheorem{defin}[teo]{Definition}

\newtheorem{oss}[teo]{Remark}
\endgroup

\numberwithin{equation}{section}

\newcommand{\myintol}[1]{\int_0^L{#1} \,dx_1}
\newcommand{\myintom}[1]{\int_{\Omega}{#1 \,dx}}
 
\newcommand{\myintegra}[1]{\int_S{#1 \,dx_2dx_3}}

\newcommand{\deb}[0]{\rightharpoonup}

\newcommand{\til}[1]{\widetilde{#1}}

\newcommand{\lin}[1]{\overline{#1}}
\newcommand{\capp}[1]{\widehat{#1}}
\newcommand{\sym}{\mathrm{sym}}
\renewcommand{\div}{\mathrm{div}}
\newcommand{\dist}{\mathrm{dist}}
\newcommand{\comp}{{\,\circ\,}}
\newcommand{\R}{\mathbb R}
\newcommand{\mthree}{{\mathbb M}^{3{\times}3}}
\newcommand{\nablah}{\nabla_{\!h}}
\newcommand{\wto}{\rightharpoonup}

\title[Convergence of equilibria of thin elastic rods]{Convergence of equilibria of thin elastic rods under physical
growth conditions\\ 
for the energy density} 
 
\author[E. Davoli]{Elisa Davoli} 
\author[M.G. Mora]{Maria Giovanna Mora}

\address[E. Davoli]{Scuola Internazionale Superiore di Studi Avanzati, via Bonomea 265, 34136 Trieste (Italy)}
\email{davoli@sissa.it}
\address[M.G. Mora]{Scuola Internazionale Superiore di Studi Avanzati, via Bonomea 265, 34136 Trieste (Italy)}
\email{mora@sissa.it}

\subjclass[2000]{74K10, 74B20, 74G10}
\keywords{Nonlinear elasticity, rod theories, equilibrium configurations, 
stationary points}
 
\begin{document}
 
\begin{abstract}
The subject of this paper is the study of the asymptotic behaviour of the equilibrium configurations of a nonlinearly elastic thin rod, as the diameter of the
cross-section tends to zero. Convergence results are established assuming physical growth conditions for the elastic energy density and suitable scalings of the applied loads, that correspond at the limit to different rod models: the constrained linear theory, the analogous of von K\'arm\'an plate theory for rods, 
and the linear theory. 
\end{abstract} 
 
\maketitle

\section{Introduction and Statement of the Main Result}

A classical question in nonlinear elasticity is the derivation 
of lower dimensional models for thin structures (such as plates, shells, or beams) 
starting from the three-dimensional theory. 
In recent years this problem has been approached by means of $\Gamma$-convergence. 
This method guarantees, roughly speaking, the convergence of minimizers 
of the three-dimensional energy to minimizers of the deduced models. 
In this paper we discuss the convergence of three-dimensional stationary points, 
which are not necessarily minimizers, assuming physical growth conditions on the stored-energy density. 
In particular, we extend the recent results of \cite{M-S} to the case of a three-dimensional thin beam 
with a cross-section of diameter $h$ and subject to an applied normal body force of order
$h^\alpha$, $\alpha>2$. These scalings correspond at the limit to the constrained linear rod theory
($2<\alpha<3$), the analogous of von K\'arm\'an plate theory for rods ($\alpha=3$), 
and the linear rod theory ($\alpha>3$).

\

We first review the main results of the variational approach.
Let $\Omega_h=(0,L){\times}hS$ be the reference configuration of a thin elastic 
beam, where $L>0$, $S\subset\R^2$ is a bounded domain with Lipschitz boundary, 
and $h>0$ is a small parameter. Without loss of generality we shall assume that 
the two-dimensional Lebesgue measure of $S$ is equal to $1$ and
\begin{equation}\label{ipsimm}
\int_Sx_2\,dx_2dx_3=\int_Sx_3\,dx_2dx_3= \int_S x_2x_3\,dx_2dx_3=0.
\end{equation}
Let $f^h\in L^2(\Omega_h;\R^3)$ be an external body force applied to the beam.
Given a deformation $v\in W^{1,2}(\Omega_h,\R^3)$ the
total energy per unit cross-section associated to $v$ 
is defined as
$$
\cal F^h(v)=\frac{1}{h^2}\int_{\Omega_h} W(\nabla v)\, dx - \frac{1}{h^2}\int_{\Omega_h} f^h{\,\cdot\,}v\, dx,
$$
where the stored-energy density $W:\mthree\to[0,+\infty]$ is assumed to satisfy
the following natural conditions:
\begin{itemize}
\medskip
\item[(H1)] $W$ is of class $C^1$ on $\mthree_+$;
\medskip
\item[(H2)] $W(F)=+\infty$ if $\det F\leq 0$, \ $W(F)\to+\infty$ if $\det F\to 0^+$;
\medskip
\item[(H3)] $W(RF)=W(F)$ for every $R\in SO(3)$, $F\in\mthree$
(frame indifference);
\medskip
\item[(H4)] $W=0$ on $SO(3)$;
\medskip
\item[(H5)] $\exists C>0$ such that $W(F)\geq C\, \dist^2(F,SO(3))$ for every $F\in\mthree$;
\medskip
\item[(H6)] $W$ is of class $C^2$ in a neighbourhood of $SO(3)$.
\medskip
\end{itemize}
Here $\mthree_+=\{F\in\mthree: \det F>0\}$ and $SO(3)=\{R\in\mthree:R^TR=Id,\ \det R=1\}$. 
In particular, condition (H2) is related to non-interpenetration of matter 
and ensures local invertibility of $C^1$ deformations with finite energy. 

The study of the asymptotic behaviour of global minimizers of $\cal F^h$, as $h\to0$,
can be performed through the analysis of the $\Gamma$-limit of $\cal F^h$ (see
\cite{DM} for an introduction to $\Gamma$-convergence). To do this,
it is convenient to rescale $\Omega_h$ to the domain $\Omega=(0,L){\times}S$ 
and to rescale deformations according to this change of variables
by setting 
$$
y(x):=v(x_1, hx_2, hx_3)
$$
for every $x\in\Omega$. Assuming for simplicity that $f^h(x)=f^h(x_1)$,
the energy functional can be written as
$$
\cal F^h(v)=\cal J^h(y)=\int_\Omega W(\nablah y)\, dx - \int_\Omega f^h{\,\cdot\,}y\, dx,
$$
where we have used the notation
$$
\nablah y:=\Big(\partial_1 y\,\Big|\,\frac{\partial_2 y}{h}\,\Big|\,\frac{\partial_3
y}{h}\Big).
$$
Let now $y^h$ be a global minimizer of $\cal J^h$ subject to the
boundary condition
\begin{equation}\label{bdary}
y^h(0,x_2,x_3)=(0,hx_2,hx_3) \quad \text{for every }(x_2,x_3)\in S.
\end{equation}
The asymptotic behaviour of $y^h$, as $h\to0$, depends on the scaling of the
applied load $f^h$ in terms of $h$. More precisely, if $f^h$ is of order $h^\alpha$
with $\alpha\geq0$, then $\cal J^h(y^h)=O(h^\beta)$, where $\beta=\alpha$
for $0\leq\alpha\leq2$ and $\beta=2\alpha-2$ for $\alpha>2$, and $y^h$ converges
in a suitable sense to a minimizer of the $\Gamma$-limit of the rescaled functionals
$h^{-\beta}\cal J^h$, as $h\to0$ (see \cite{F-J-M2,A-B-P,M-M3,M-M,S0,S}).
In particular, it has been proved in \cite{M-M,S} that, if $f^h$ is a normal
force of the form $h^\alpha(f_2 e_2 +f_3 e_3)$, with $\alpha>2$ and
$f_2,f_3\in L^2(0,L)$, then 
$$
y^h\to x_1e_1\quad \text{in } W^{1,2}(\Omega,\R^3).
$$
In other words, minimizers converge to the identity deformation on the mid-fiber of the rod.
This suggests to introduce the (averaged) tangential and normal displacements,
respectively given by 
\begin{eqnarray}
\label{defuh}
& &
u^h(x_1):=\begin{cases}
\displaystyle
\frac{1}{h^{\alpha-1}}
\myintegra{(y^h_1-x_1)} & \ \text{if }
\alpha\geq3,
\medskip
\\ 
\displaystyle
\frac{1}{h^{2(\alpha-2)}}
\myintegra{(y^h_1-x_1)} & \ \text{if }
2<\alpha<3,
\end{cases} 
\\
\nonumber\\
\label{defvhk}
& &
v^h_k(x_1):=\frac{1}{h^{\alpha-2}}\myintegra{y^h_k} \quad \text{for } k=2,3
\end{eqnarray}
for a.e.\ $x_1\in(0,L)$,
and the (averaged) twist function, given by 
\begin{equation}\label{defwh}
w^h(x_1):=\frac{1}{\mu(S)}\frac{1}{h^{
\alpha-1}}\myintegra{(x_2 y^h_3-x_3 y^h_2)}
\end{equation}
for a.e.\ $x_1\in(0,L)$, where $\mu(S):=\myintegra{(x_2^2+x_3^2)}$.
As $h\to0$, the sequence $(u^h, v^h_2, v^h_3, w^h)$ converges
strongly in $W^{1,2}$ to a limit $(u,v_2,v_3,w)$, which is a global minimizer of
the functional $\cal J_\alpha$ given by the $\Gamma$-limit of
$h^{-2\alpha+2}\cal J^h$.
If $\alpha=3$, the $\Gamma$-limit $\cal J_3$ corresponds to the
one-dimensional analogous of the von K\'arm\'an plate functional. 
For $\alpha>3$ the functional $\cal J_\alpha$ coincides with the linear
rod functional, while for $2<\alpha<3$ the limiting energy is still linear but is
subject to a nonlinear isometric constraint (see Section~\ref{sec:pr} for the exact
definition of the functionals $\cal J_\alpha$).

\

In this paper we focus on the study of the asymptotic behaviour of (possibly non minimizing)
stationary points of $\cal J^h$, as $h\to0$. The first convergence results 
for stationary points have been proved in \cite{M-M2,M-M-S,M-P}. 
We also point out the recent results \cite{AMM1,AMM2} concerning the dynamical case.
A crucial assumption in all these papers is that the stored-energy function $W$
is everywhere differentiable and its derivative satisfies a linear growth condition.
Unfortunately, this requirement is incompatible with the physical assumption
(H2). At the same time, if (H2) is satisfied, the conventional form of the
Euler-Lagrange equations of $\cal J^h$ is not well defined and it is
not even clear to which extent minimizers of $\cal J^h$ satisfy this condition
(we refer to \cite{B} and \cite{M-S} for a more detailed discussion). 

Following \cite{M-S}, we consider an alternative first order stationarity condition, 
introduced by Ball in \cite{B}. To this aim we require the following additional assumption:
\begin{itemize}
\smallskip
\item[(H7)] $\exists k>0$ such that $|DW(F)F^T|\leq k(W(F)+1)$ for every $F\in\mthree_+$.
\smallskip
\end{itemize}
This growth condition is compatible with (H1)--(H6) (see \cite{B}). 

\begin{defin}\label{pstazball}
We say that a deformation $y\in W^{1,2}(\Omega,\R^3)$ is a stationary
point of $\cal J^h$ if it satisfies the boundary condition 
$y(0,x_2,x_3)=(0,hx_2,hx_3)$ for every $(x_2,x_3)\in S$ and the equation
\begin{equation}\label{pstcresc}
\myintom{DW(\nablah y)(\nablah y)^T{\,:\,}
[(\nabla \phi)\comp y]}=\myintom{f^h{\,\cdot\,}(\phi\comp y)}
\end{equation}
for every $\phi\in C^1_b(\R^3,\R^3)$ such that
$\phi(0,hx_2,hx_3)=0$ for all $(x_2,x_3)\in S$.
\end{defin}

In the previous definition and in the sequel $C^1_b(\R^3,\R^3)$
denotes the space of $C^1$ functions that are bounded in $\R^3$, with
bounded first-order derivatives.

Assuming (H1)--(H7) and using external variations of the form $y+\epsilon\phi\comp y$,  
one can show that every local minimizer $y$ of $\cal J^h$, subject to the boundary condition
$y(0, x_2,x_3)=(0, hx_2,hx_3)$ for every $(x_2,x_3)\in S$, is a stationary
point of $\cal J^h$ in the sense of Definition~\ref{pstazball} (\cite[Theorem~2.4]{B}).
Moreover, when minimizers are invertible, condition \eqref{pstcresc} corresponds
to the equilibrium equation for the Cauchy stress tensor.

In \cite{M-S} it has been proved that
stationary points in the sense of Definition~\ref{pstazball} converge
to stationary points of the $\Gamma$-limit $\cal J_\alpha$ in the case of a thin plate 
and for the scaling $\alpha\geq 3$ (corresponding to von K\'arm\'an and to linear plate theory).
In this paper we extend this result to the range of scalings $\alpha>2$
in the case of a thin beam.  
Our main result is the following.

\begin{teo}\label{psstball}
Assume that $W$ satisfies (H1)--(H7).
Let $f_2$, $f_3\in L^2(0,L)$, and ${\alpha>2}$.
For every $h>0$ let $y^h$ be a stationary point
of $\cal J^h$ 
(according to Definition~\ref{pstazball}) with $f^h:=h^\alpha(f_2e_2+f_3e_3)$. 
Assume there exists $C>0$ such that 
\begin{equation}\label{energiapiccola2}
\myintom{W(\nablah y^h)}\leq Ch^{2\alpha-2}
\end{equation}
for every $h>0$. Then,
\begin{equation}\label{yhconv}
y^h\to x_1 e_1\quad \text{in } W^{1,2}(\Omega,\R^3).
\end{equation}
Moreover, let $u^h$, $v^h$, and $w^h$ be the scaled displacements and twist
function, introduced in (\ref{defuh})--(\ref{defwh}). 
Then, up to subsequences, we have
\begin{eqnarray*}
& u^h\wto u & \text{in }W^{1,2}(0,L),\\
& v^h_k\to v_k & \text{in }W^{1,2}(0,L) \quad \text{for }
k=2,3,\\ 
& w^h\wto w & \text{in }W^{1,2}(0,L),
\end{eqnarray*}
where $(u,v_2,v_3,w)\in W^{1,2}(0,L){\times}
W^{2,2}(0,L){\times}W^{2,2}(0,L){\times}W^{1,2}(0,L)$
is a stationary point of $\cal J_{\alpha}$.
\end{teo}

The proof of Theorem~\ref{psstball} is closely related to \cite{M-S} and uses
as key tool the rigidity estimate proved in \cite{F-J-M}.
The main new idea with respect to \cite{M-S} is the construction
of a sequence of suitable ``approximate inverse functions'' of the
deformations $y^h$ (see Lemma~\ref{succappr}), which allows us to extend
the results of \cite{M-S} to the range of scalings $\alpha\in(2,3)$.
This construction is based on a careful study of the asymptotic
development of the deformations $y^h$ in terms of approximate displacements
and uses in a crucial way the fact that the limit space dimension is 
equal to one.

\section{Preliminary Results}\label{sec:pr}

In this section we recall the expression of the
$\Gamma$-limits $\cal J_\alpha$ identified in \cite{M-M} and \cite{S}
and we prove some preliminary results.

We start by introducing some notation.
Let $Q_3:\mthree\to[0,+\infty)$ be the quadratic form of linearized elasticity:
$$
Q_3(F):=D^2W(Id)F{\,:\,}F \quad \textrm{for every }F \in\mthree.
$$
We will denote by $\cal L$ the associated linear map on $\mthree$ given by
$\cal L:=D^2W(Id)$.
Let
\begin{equation}\label{EE}
\mathbb{E}:=\min_{a,b\in \R^3} Q_3(e_1|a|b),
\end{equation}
and let $Q_1$ be the quadratic form defined on the space $\mthree_{skew}$ of
skew-symmetric matrices given by
\begin{equation}\label{QQ}
Q_1(F):=\min_{\beta\in W^{1,2}(S,\R^3)}
\int_S Q_3\Big( x_2Fe_2+x_3Fe_3\,\Big|\,\partial_2\beta\,\Big|\,\partial_3 \beta\Big)\,
dx_2dx_3
\end{equation}
for every $F\in \mthree_{skew}$. 
It is easy to deduce from the assumptions (H1)--(H6) that
$\mathbb E$ is a positive constant and $Q_1$ is a positive
definite quadratic form.

The functionals $\cal J_\alpha$ are defined on the space 
$$
H:=W^{1,2}(0,L){\times}W^{2,2}(0,L){\times}
W^{2,2}(0,L){\times} W^{1,2}(0,L)
$$ 
and are finite on
the class $\mathcal A_\alpha$, which can be described as follows:
\begin{multline}\nonumber
\cal A_\alpha:=\big\{(u,v_2,v_3,w)\in H: \ u'+\tfrac12[(v_2')^2+(v_3')^2]=0 \text{ in } (0,L)
\textrm{ and}
\smallskip
\\
u(0)=v_k(0)=v_k'(0)=w(0)=0 \text{ for
}k=2,3\big\}
\end{multline}
for $2<\alpha<3$, and 
$$
\cal A_{\alpha}:=\big\{(u,v_2,v_3,w)\in H: \
u(0)=v_k(0)=v_k'(0)=w(0)=0 \text{ for }k=2,3\big\}
$$
for $\alpha\geq 3$.

For $2<\alpha<3$ the functional $\cal J_\alpha$ is given by
\begin{equation}\label{i<3}
\mathcal J_{\alpha}(u,v_2,v_3,w) =
\frac{1}{2}\myintol{Q_1(A')}-\myintol{(f_2v_2+f_3v_3)}
\end{equation}
for every $(u,v_2,v_3,w)\in\cal A_\alpha$, $\mathcal J_{\alpha}(u,v_2,v_3,w) =+\infty$
otherwise in $H$. In \eqref{i<3} the function
$A\in W^{1,2}((0,L),\mthree)$ is defined by
\begin{equation}\label{A-def}
A(x_1):=\left(\begin{array}{ccc}
0 & -v_2'(x_1) & -v_3'(x_1) \\
v_2'(x_1) & 0 & -w(x_1) \\
v_3'(x_1) & w(x_1) & 0
\end{array}\right)
\end{equation}
for a.e.\ $x_1\in(0,L)$.

For $\alpha=3$ the $\Gamma$-limit is given by
\begin{align}\nonumber
\cal J_3 (u,v_2,v_3,w)=
\frac12 \int_0^L \mathbb{E}\,\big(&u'+\tfrac12[(v'_2)^2+(v'_3)^2]\big)^2\, dx_1
\\
\label{i3}
& {}+ \frac{1}{2}\myintol{Q_1(A')} -\myintol{(f_2v_2+f_3v_3)}
\end{align}
for every $(u,v_2,v_3,w)\in\cal A_\alpha$, $\mathcal J_3(u,v_2,v_3,w) =+\infty$
otherwise in $H$.

Finally, for $\alpha>3$ the $\Gamma$-limit is given by 
\begin{align}\nonumber
\cal J_{\alpha}(u,v_2,v_3,w) =
\frac{1}{2}\int_0^L\mathbb{E}\,(u')^2\, dx_1 & + \frac{1}{2}\myintol{Q_1(A')}
\\
\label{i>3} &{}-\myintol{
(f_2v_2+f_3v_3)}
\end{align}
for every $(u,v_2,v_3,w)\in\cal A_\alpha$, $\mathcal J_{\alpha}(u,v_2,v_3,w) =+\infty$
otherwise in $H$.

We can now compute the Euler-Lagrange equations for the
functionals $\cal J_{\alpha}$ introduced above. 
We first recall the following lemma.

\begin{lem}\label{lem3}
Let $F\in\mthree_{skew}$ and let 
$\cal G_F:W^{1,2}(S,\R^3)\to [0,+\infty)$ be the functional   
$$
\cal G_F(\beta):=\int_S Q_3\Big(x_2Fe_2+x_3Fe_3\,\Big|\, \partial_2\beta \,\Big| 
\,\partial_3 \beta\Big)\, dx_2dx_3
$$
for every $\beta\in W^{1,2}(S,\R^3)$. Then $\cal G_F$ is convex and has a
unique minimizer in the class
$$ 
\cal B: =\Big\{\beta\in W^{1,2}(S,\R^3):\myintegra{\beta}=\myintegra{\partial_2
\beta}=\myintegra{\partial_3 \beta}=0 \Big\}.
$$ 
Furthermore, a function $\beta\in\cal B$ is the minimizer of $\cal G_F$ if and
only if the map $E:S\to\mthree$ defined by
\begin{equation}
E:=\cal L\Big(x_2Fe_2+x_3Fe_3\,\Big|\,\partial_2\beta\,\Big|\,\partial_3 \beta\Big) 
\end{equation}
satisfies in a weak sense the following problem:
$$
\begin{cases}
\div_{x_2,x_3}(Ee_2|Ee_3)=0 & \text{ in } S,
\smallskip\\
(Ee_2|Ee_3)\nu_{\partial S}=0 & \text{ on } \partial S,
\end{cases}
$$
where $\nu_{\partial S}$ is the unit normal to $\partial S$. Finally,
the minimizer depends linearly on $F$.
\end{lem}

\begin{proof}
See \cite[Lemma 2.1]{M-M2} and \cite[Remark 3.4]{M-M3}. 
\end{proof}

We shall use the following notation: for each $F\in L^1(\Omega,\mthree)$ 
we define the zeroth order moment of $F$ as the function 
$\lin{F}:(0,L)\to \mthree$ given by
$$
\lin{F}(x_1):=\myintegra{F(x)}
$$
for a.e.\ $x_1\in(0,L)$. We also introduce the first order moments of
$F$ as the functions $\til{F},\capp{F}:(0,L)\to\mthree$ given by
$$
\til{F}(x_1):=\myintegra{x_2F(x)}, \qquad
\capp{F}(x_1)=\myintegra{x_3F(x)}
$$
for a.e.\ $x_1\in(0,L)$.

The following proposition follows now from straightforward computations.

\begin{prop}
Let $(u,v_2,v_3,w)\in \cal A_{\alpha}$. For a.e.\ $x_1\in(0,L)$ 
let $\beta(x_1,\cdot,\cdot)\in \cal B$ be the minimizer of $\cal G_{A'(x_1)}$, 
where $A'$ is the derivative of the function $A$ introduced in \eqref{A-def}.
Let also $E:\Omega\to\mthree$ be defined by
$$
E:=\cal L \Big( x_2A'e_2+x_3A'e_3\,\Big|\,\partial_2\beta
\,\Big|\,\partial_3 \beta\Big),
$$
and let $\til{E}$ and $\capp{E}$ be its first order moments.
Then 
\begin{enumerate}[(1)]
\item $(u,v_2,v_3,w)$ is a stationary point of $\cal J_3$ if and only if 
the following equations are satisfied:
\begin{align}
& \label{eq1a}
u'+\frac{1}{2}[(v_2')^2+(v_3')^2]=0 \quad \text{ in } (0,L), 
\medskip\\
& \label{eq2a}
\begin{cases}
\til{E}''_{11}+f_2=0 \quad \text{in } (0,L),
\smallskip\\
\til{E}_{11}(L)=\til{E}'_{11}(L)=0, 
\end{cases}
\\
& \label{eq2b}
\begin{cases}
\capp{E}''_{11}+f_3=0 \quad \text{in }(0,L),
\smallskip\\ 
\capp{E}_{11}(L)=\capp{E}'_{11}(L)=0, 
\end{cases}
\\
&\label{eq3}
\begin{cases}
\til{E}'_{12}=\capp{E}'_{13} \quad \text{in }(0,L),
\smallskip\\ 
\til{E}_{12}(L)=\capp{E}_{13}(L);
\end{cases}
\end{align}
\item if $\alpha>3$, $(u,v_2,v_3,w)$ is a stationary point of $\cal J_{\alpha}$ 
if and only if 
\begin{equation} \label{eq1b} 
u'=0\quad \text{ in } (0,L)
\end{equation}
and \eqref{eq2a}--\,\eqref{eq3} are satisfied; 
\item if $2<\alpha<3$, $(u,v_2,v_3,w)$ is a stationary point of $\cal
J_{\alpha}$ 
if and only if \eqref{eq2a}--\,\eqref{eq3} are satisfied.
 \end{enumerate}
\end{prop}
 
\begin{oss}
If $(u,v_2,v_3,w)\in \cal A_{\alpha}$ and $2<\alpha<3$, then
$u$ is uniquely determined in terms of $v_2$ and $v_3$. Indeed by the
constraint
 $$u'+\frac{(v_2')^2+(v_3')^2}{2}=0 \textrm{ a.e. in }(0,L)$$
and the boundary condition $u(0)=0$, we have
\begin{equation}\label{ux1}
u(x_1)=-\int_0^{x_1}{\frac{(v_2')^2+(v_3')^2}{2}}\quad\textrm{ for a.e. }
x_1 \textrm{ in
}(0,L).
\end{equation}
For $\alpha\geq3$ the same conclusion holds when $(u,v_2,v_3,w)\in \cal A_{\alpha}$
is a stationary point of $\cal J_{\alpha}$. Indeed, if $\alpha=3$, (\ref{eq1a}) yields 
\eqref{ux1}, while, if $\alpha>3$, (\ref{eq1b}) gives 
$$
u=0\textrm{ a.e. in }(0,L).
$$
Using the previous observations and the strict convexity of $Q_1$, it is easy to show that
for every $\alpha>2$, $\cal J_{\alpha}$ has a unique stationary
point that is a minimum point.
\end{oss}

\begin{oss}
For what concerns the three-dimensional functionals $\cal J^h$, under additional hypotheses on $W$ (such as polyconvexity, see \cite{B1}) it is possible to show existence of global minimizers, and therefore of stationary points. Furthermore, they automatically satisfy the energy estimate \eqref{energiapiccola2} (see \cite[proof of Theorem~2]{F-J-M2}). For general $W$ the existence of stationary points (according to Definition~\ref{pstcresc} or to the classical formulation) is a subtle issue. We refer to \cite[Section~2.7]{B} for a discussion of results in this direction.
\end{oss}

From now on we shall work with sequences of deformations $y^h\in
W^{1,2}(\Omega,\R^3)$, satisfying the boundary condition \eqref{bdary}
and the uniform energy estimate \eqref{energiapiccola2} with $\alpha>2$.
This bound, combined with the coercivity condition (H5), 
provides us with a control on the distance of $\nablah y^h$ from $SO(3)$. 
This fact, together with the geometric
rigidity estimate by Friesecke, James and M\"{u}ller \cite[Theorem 3.1]{F-J-M}, 
allows us to construct an approximating sequence
of rotations $(R^h)$, whose $L^2$-distance from $\nablah y^h$ is of the same
order in terms of $h$ of the
$L^2$-norm of $\dist(\nablah y^h, SO(3))$. More precisely, the following result
holds
true.

\begin{teo}\label{rot}
Assume that $W:\mthree\to[0,+\infty]$ is continuous and
satisfies (H3)--(H6).
Let $\alpha>2$ and let $(y^h)$ be a sequence in $W^{1,2}(\Omega,\R^3)$
satisfying \eqref{bdary} and \eqref{energiapiccola2} for every $h>0$.
Then there exists a sequence 
$(R^h)$ in $C^\infty((0,L),\mthree)$ such that
\begin{eqnarray}
\label{c1}& R^h(x_1)\in SO(3)\quad \text{for every } x_1 \in (0,L),\\
\label{c2}& \|\nablah y^h-R^h\|_{L^2}\leq Ch^{\alpha-1},\\ 
\label{der} &\|(R^h)'\|_{L^2}\leq Ch^{\alpha-2}, \\
\label{linfnorm} & \|R^h-Id\|_{L^\infty}\leq Ch^{\alpha-2}.
\end{eqnarray}
\end{teo}

We omit the proof as it follows closely the proof of
\cite[Proposition~4.1]{M-M-S}.
Owing to the previous approximation result, one can deduce the following
compactness properties.

\begin{teo}\label{teoa}
Under the assumptions of Theorem~\ref{rot}, let
$u^h$, $v^h_2$, $v^h_3$, $w^h$ be the scaled displacements and twist function
introduced in (\ref{defuh})--(\ref{defwh}).
Then
\begin{equation}\label{convyh}
y^h\to x_1e_1\quad \text{strongly in }W^{1,2}(\Omega,\R^3)
\end{equation}
and there exists $(u,v_2,v_3,w)\in \cal A_{\alpha}$ such that, up to
subsequences, we have
\begin{eqnarray}
& &u^h\to u \quad \text{strongly in }W^{1,2}(0,L)\quad \text{if }
2<\alpha<3,
\label{umin}
\\
& &u^h\deb u \quad \text{weakly in }W^{1,2}(0,L)\quad \text{if }
\alpha \geq 3,
\label{u}
\\
& &v^h_k\to v_k\quad \text{strongly in }W^{1,2}(0,L), \quad k=2,3,
\label{v}\\
& &w^h\deb w\quad \text{weakly in }W^{1,2}(0,L).\label{w}
\end{eqnarray}
Moreover, let $A \in W^{1,2}((0,L),\mthree)$ be the function defined in \eqref{A-def}.
Then, if $R^h$ is the approximating sequence of rotations given by
Theorem~\ref{rot}, the following convergence properties hold true:
\begin{eqnarray}
\label{grad-Id}& &\frac{\nablah{y}^h-Id}{h^{\alpha-2}}\to A \quad
\text{strongly in }
L^2(\Omega, \mthree),
\\
\label{ah}& &A^h:=\frac{R^h-Id}{h^{\alpha-2}}\deb A \quad 
\text{weakly in } W^{1,2}((0,L), \mthree),
\\
\label{syma}& &\frac{\sym (R^h-Id)}{h^{2(\alpha-2)}}\rightarrow \frac{A^2}{2}
\quad 
\text{uniformly in } (0,L).
\end{eqnarray}
\end{teo}

For the proof we refer to \mbox{\cite[Theorem 3.3]{S}}.

\

We conclude this section by proving a lemma, which will be crucial to extend
the convergence of equilibria result to the scalings $\alpha\in(2,3)$.

\begin{lem}\label{succappr}
Under the assumptions of Theorem~\ref{rot}, there exist two
sequences $(\xi^h_k)$, $k=2,3$, such that for every $h>0$
\begin{eqnarray}
\label{bordoxih}& &\xi^h_k\in C^1_b(\R), \quad \xi^h_k(0)=0,
\vphantom{\frac1h}\\
\label{convxih}& &\frac{y^h_k}{h}-\frac{1}{h}\xi_k^h\comp y^h_1\to x_k \quad
\textrm{strongly in }L^2(\Omega),
\\
\label{limitazionenorme}& &
\|\xi^h_k\|_{L^{\infty}}+\|(\xi^h_k)'\|_{L^{\infty}}\leq Ch^{\alpha-2}.
\end{eqnarray} 
\end{lem}

\begin{oss}\label{invappr}
The sequences $(\xi^h_k)$ of the previous lemma can be interpreted as follows: the functions defined by
$$
\omega^h(x)=\Big(x_1,\frac{x_2}{h}-\frac{\xi^h_2(x_1)}{h},\frac{x_3
}{h}-\frac{\xi^h_3(x_1)}{h}\Big)
$$
represent a sort of ``approximate inverse functions'' of the deformations
$y^h$, in the sense that the compositions
$\omega^h\comp y^h$ converge to the identity strongly in $L^2(\Omega,\R^3)$
by \eqref{convyh} and \eqref{convxih}.
\end{oss}

\begin{proof}[Proof of Lemma~\ref{succappr}]
In order to construct the functions $\xi^h_k$, we first study the asymptotic
behaviour of the sequences $(\frac{1}{h}y^h_k)$, $k=2,3.$
By Poincar\'e inequality we obtain the estimate
$$
\Big\|\frac{y^h_k}{h}-x_k-\myintegra{\Big(\frac{y^h_k}{h}-x_k\Big)}
\Big\|_{L^2}\leq C\Big(\Big\|\frac{\partial_k
y^h_k}{h}-1\Big\|_{L^2}+\Big\|\frac{\partial_j
y^h_k}{h}\Big\|_{L^2} \Big),
$$
where $k,j\in\{2,3\}$, $k\neq j$.
Therefore, by \eqref{defvhk} and \eqref{grad-Id} we have
\begin{equation}\label{convx2x3}
 \Big\|\frac{y^h_k}{h}-x_k-h^{\alpha-3}v^h_k\Big\|_{L^2(\Omega)}
\leq Ch^{\alpha-2}.
\end{equation} 
In particular, for $\alpha>3$ it follows that $y^h_k\to x_k$ strongly in
$L^2$, so that, if $\alpha>3$, we can simply take $\xi^h_k=0$ for $k=2,3$ and
every $h>0$. If $2<\alpha\leq 3$, we need to construct a suitable
approximation of $v^h_k$. 
Let $(R^h)$ be the approximating sequence of rotations associated with $(y^h)$
(see Theorems~\ref{rot} and~\ref{teoa}). By (\ref{linfnorm}) and (\ref{syma}) we
deduce the following estimates: 
\begin{eqnarray*}
& &\|R^h_{k1}\|_{L^{\infty}}\leq Ch^{\alpha-2}\quad\textrm{ for }k=2,3,\\  
& &\|R^h_{11}-1\|_{L^{\infty}}\leq Ch^{2(\alpha-2)}.
\end{eqnarray*}
Let $r^h_{k},r^h_{1}\in C(\R)$ be continuous extensions
of the functions $R^h_{k1}$ and $R^h_{11}-1$ to $\R$
such that for every $h>0$
\begin{eqnarray}
 \label{rh0}& &\textrm{supp }r^h_{k},\textrm{ supp }r^h_{1}\subset(-1,L+1),\\
 \label{rh1}& &r^h_{k}=R^h_{k1} \textrm{ in }(0,L)\quad\textrm{ for }k=2,3,\\
 \label{rh2}& &r^h_{1}=R^h_{11}-1\textrm{ in }(0,L),\\
 \label{rh3}& &\|r^h_{k}\|_{L^{\infty}}\leq Ch^{\alpha-2}\quad\textrm{ for
}k=2,3,\\
 \label{rh4}& &\|r^h_{1}\|_{L^{\infty}}\leq Ch^{2(\alpha-2)}.
\end{eqnarray}
We introduce the functions $\tilde v^h_1$, $\tilde v^h_k\in
C^1_b(\R)$ defined by 
\begin{eqnarray}\label{tilvk}
\tilde v^h_k(x_1) & := &\int_0^{x_1} r^h_k(s)\,ds,
\\ 
\label{tilv1}
\tilde v^h_1(x_1) & := & \int_0^{x_1}r^h_{1}(s)\,ds.
\end{eqnarray}
Using the boundary condition \eqref{bdary}, the Poincar\'e inequality, (\ref{c2}), and 
(\ref{grad-Id}), we obtain
\begin{equation}\label{y3hrot}
\Big\|\frac{y^h_k}{h}-x_k-\frac{1}{h}\tilde
v^h_k\Big\|_{L^2}\leq Ch^{\alpha-2},
\end{equation}
and analogously,
\begin{equation}\label{y1hrot}
\|y^h_1-x_1-\tilde v^h_1\|_{L^2}\leq Ch^{\alpha-1}.
\end{equation}
This last inequality, together with (\ref{rh4}), implies that 
\begin{equation}\label{y1x1}
 \|y^h_1-x_1\|_{L^2}\leq Ch^{2(\alpha-2)}\quad\textrm{ for }\alpha\leq3.
\end{equation}
We are now in a position to construct the maps $\xi^h_k$ when
$\alpha\leq3$. If $\alpha=3$, we define $\xi^h_k=\tilde v^h_k$.
Properties (\ref{bordoxih}) and (\ref{limitazionenorme}) follow
immediately. To verify (\ref{convxih}) it is enough to remark that by
(\ref{rh3}) and (\ref{y3hrot}) we have
\begin{eqnarray}
\nonumber\Big\|\frac{y^h_k}{h}-x_k-\frac{\tilde
v^h_k\comp y^h_1}{h}\Big\|_{L^2}
&\leq& Ch + \frac{1}{h}\|\tilde v^h_k\comp y^h_1-\tilde
v^h_k\|_{L^2}\\
\nonumber&\leq& Ch+\frac{1}{h}\|(\tilde
v^h_k)'\|_{L^{\infty}}\|y^h_1-x_1\|_{L^2}\leq
Ch.
\end{eqnarray}
If $2<\alpha<3$, we first fix $n_0\in\mathbb{N}$ such that 
\begin{equation}\label{fixn0}
\alpha>2+\frac{1}{2n_0+3}
\end{equation}
and we introduce a sequence of maps $(\zeta^h_n)$, $n=1,\ldots ,n_0,$
recursively defined as
\begin{equation}\label{zita}
\begin{array}{c l l}
\zeta^h_{n_0}(x_1) &\hspace{-0.2 cm}=\hspace{-0.2 cm}&  x_1-\tilde v^h_1(x_1), 
\medskip\\
\zeta^h_n(x_1) &\hspace{-0.2 cm}=\hspace{-0.2 cm}& x_1-\tilde
v^h_1\comp\zeta_{n+1}(x_1)\quad \text{ for
}n=1,\dots,n_0-1.
\end{array}
\end{equation}
For $k=2,3$ and every $h>0$ we define
\begin{equation}\label{xihk}
\xi^h_k:=\tilde v^h_k\comp\zeta^h_1.
\end{equation}
Since $\zeta^h_{n_0}(0)=0$, we have by induction that $\zeta^h_n(0)=0$ for each
$n=1,2,\ldots,n_0$, so that $\xi^h_k(0)=0$. From the
regularity of $\tilde v^h_1$ and $\tilde v^h_k$ it follows that (\ref{bordoxih})
is satisfied. By (\ref{rh3}) we deduce
\begin{equation}\label{normainftyxih}
 \|\xi^h_k\|_{L^{\infty}}\leq \|\tilde
v^h_k\|_{L^{\infty}}\leq (L+2)
\|r^h_{k}\|_{L^{\infty}}\leq Ch^{\alpha-2}.
\end{equation}
To estimate $\|(\xi^h_k)'\|_{L^{\infty}}$, we first deduce a
recursive bound
for $\|(\zeta^h_n)'\|_{L^{\infty}}$. If $h$ is small
enough, we have
$$
\|(\tilde v^h_1)'\|_{L^{\infty}}\leq 1.
$$ 
By (\ref{zita}) the following inequalities hold true:
\begin{eqnarray}
&&\|(\zeta^h_{n_0})'\|_{L^{\infty}}  \leq  
1+\|(\tilde v^h_1)'\|_{L^{\infty}}\leq 2,
\\
&&\|(\zeta^h_{n})'\|_{L^{\infty}} \leq 
1+\|(\zeta_{n+1}^h)'\|_{L^{\infty}} \quad 
\text{for }n=1,\dots,n_0-1,
\\
&&\label{zita1'} \|(\zeta^h_1)'\|_{L^{\infty}} \leq
1+n_0.
\end{eqnarray}
Now by (\ref{zita1'}) and (\ref{rh3}) we have 
\begin{equation}\label{normainftyder}
\|(\xi^h_k)'\|_{L^{\infty}}\leq \|(\tilde
v^h_k)'\|_{L^{\infty}}\|(\zeta_1^h)'\|_{L^{\infty}}\leq
(1+n_0)\|r^h_{k}\|_{L^{\infty}}\leq Ch^{\alpha-2}.
\end{equation}
Combining (\ref{normainftyxih}) and (\ref{normainftyder}) we
obtain (\ref{limitazionenorme}).
To conclude the proof it remains to verify (\ref{convxih}). By (\ref{rh4}),
(\ref{y1hrot}), and (\ref{y1x1}) we have
\begin{eqnarray}
\hspace{-0.6 cm}\|\zeta^h_{n_0}\comp y^h_1-x_1\|_{L^2} &\hspace{-0.2
cm}=&\hspace{-0.2 cm} 
\|y^h_1-\tilde v^h_1\comp y^h_1-x_1\|_{L^2}
\nonumber
\\
&\hspace{-0.2 cm}\leq&\hspace{-0.2 cm}
\|y^h_1-\tilde v^h_1-x_1\|_{L^2}
+ \|\tilde v^h_1-\tilde v^h_1\comp y^h_1\|_{L^2}
\nonumber 
\\
&\hspace{-0.2 cm}\leq&\hspace{-0.2 cm} Ch^{\alpha-1} + \|(\tilde
v^h_1)'\|_{L^{\infty}}
\|y^h_1-x_1\|_{L^2}
\nonumber 
\\
\label{varzetan0}
&\hspace{-0.2 cm}\leq&\hspace{-0.2 cm} Ch^{\alpha-1}+
h^{2(\alpha-2)}\|r^h_{1}\|_{L^{\infty}}\leq
Ch^{\alpha-1}+Ch^{4(\alpha-2)}.
\end{eqnarray}
Arguing analogously for $\zeta^h_{n_0-1}$ and using (\ref{varzetan0}), we
obtain
\begin{eqnarray}
 \nonumber \|\zeta^h_{n_0-1}\comp y^h_1-x_1\|_{L^2}&\hspace{-0.2
cm}\leq&\hspace{-0.2 cm} \|y^h_1-x_1-\tilde v^h_1\|_{L^2}+\|\tilde
v^h_1-\tilde v^h_1\comp \zeta^h_{n_0}\comp y^h_1\|_{L^2}\\
 \nonumber &\hspace{-0.2 cm}\leq&\hspace{-0.2 cm} Ch^{\alpha-1}+\|(\tilde
v^h_1)'\|_{L^{\infty}}\|\zeta^h_{n_0}\comp y^h_1-x_1\|_{L^2}\\
 \nonumber &\hspace{-0.2 cm}\leq&\hspace{-0.2 cm}
Ch^{\alpha-1}+Ch^{2(\alpha-2)}(h^{\alpha-1}+h^{4(\alpha-2)})\\
 \label{varzn0-1} &\hspace{-0.2 cm}\leq&\hspace{-0.2 cm}
Ch^{\alpha-1}+Ch^{6(\alpha-2)}.
\end{eqnarray}
By induction we deduce
\begin{equation}\label{hpvarzetan}
\|\zeta^h_{n}\comp y^h_1-x_1\|_{L^2}\leq
Ch^{\alpha-1}+Ch^{2(n_0-n+2)(\alpha-2)}.
\end{equation}
In particular, we have
\begin{equation}\label{varzeta1}
\|\zeta^h_1\comp y^h_1-x_1\|_{L^2}\leq Ch^{\alpha-1}+Ch^{2
(n_0+1)(\alpha-2)}.
\end{equation}
We can now prove (\ref{convxih}). By (\ref{xihk}), (\ref{rh3})
and (\ref{varzeta1}) we obtain
\begin{eqnarray}
\nonumber\frac{1}{h}\|\xi^h_k\comp y^h_1-\tilde
v^h_k\|_{L^2}\hspace{-0.4 cm}&
&=\frac{1}{h}\|\tilde v^h_k\comp \zeta^h_1\comp y^h_1-\tilde
v^h_k\|_{L^2}\\
 \nonumber & &\leq \frac{1}{h}\|(\tilde
v^h_k)'\|_{L^{\infty}}\|\zeta^h_1\comp y^h_1-x_1\|_{L^2}\\
 \nonumber & &\leq
\frac{1}{h}\|r^h_{k}\|_{L^{\infty}}(Ch^{\alpha-1}+Ch^{2(n_0+1)(\alpha-2)})\\
 \nonumber & &\leq Ch^{\alpha-3}(Ch^{\alpha-1}+Ch^{2(n_0+1)(\alpha-2)})\\
 \nonumber & &\leq Ch^{\min\{2\alpha-4,(2n_0+3)\alpha-(4n_0+7)\}},
\end{eqnarray}
where the last term converges to zero because of \eqref{fixn0}. Combining this
with (\ref{y3hrot}), we deduce (\ref{convxih}).
\end{proof}

\section{Proof of the Main Result}

This section is entirely devoted to the proof of Theorem~\ref{psstball}. 
The proof strategy is similar to \cite{M-S}. The major difference is in the
analysis of the asymptotic behaviour of the first-order stress moments (Steps~6
and~7), where the approximating sequences constructed in Lemma~\ref{succappr} are
needed to define suitable test functions in the scalings $2<\alpha<3$.

\begin{proof}[Proof of Theorem~\ref{psstball}]
Let $(y^h)$ be a sequence of deformations in $W^{1,2}(\Omega,\R^3)$
satisfying the energy bound \eqref{energiapiccola2}, the boundary condition
\eqref{bdary}, and the Euler-Lagrange equations
\begin{equation}\label{2pstcresc}
\myintom{DW(\nablah y^h)(\nablah y^h)^T{\,:\,}[(\nabla
\phi)\comp y^h]}=\myintom{h^{\alpha}[f_2(\phi_2\comp y^h)+f_3(\phi_3\comp y^h)]}
\end{equation}
for every $\phi \in C^1_b(\R^3,\R^3)$ 
such that $\phi(0,hx_2,hx_3)=0$ for all $(x_2,x_3)\in S$.

Convergence of the sequences $(y^h)$, $(u^h)$, $(v^h_k)$, and
$(w^h)$ 
follows from Theorem~\ref{teoa}, together with the fact that 
$(u,v_2,v_3,w)\in \cal A_{\alpha}$. 
To conclude the proof we need to show that $(u,v_2,v_3,w)$ is a
stationary point of~$\cal J_{\alpha}$.

The proof is split into seven steps.
\medskip

\noindent
\textbf{Step 1.} \textit{Decomposition of the deformation gradients in rotation
and strain}
\smallskip

\noindent
Let $(R^h)$ be the approximating sequence of rotations constructed in
Theorem~\ref{rot} and let $A\in W^{1,2}((0,L),\mthree)$ be the function defined in \eqref{A-def}.
We introduce the strain $G^h:\Omega\to\mthree$ as 
\begin{equation}\label{dec}
\nablah y^h=R^h(Id+h^{\alpha-1}G^h).
\end{equation}
By \eqref{c2} the sequence $(G^h)$ is bounded in $L^2(\Omega,\mthree)$, 
so that there exists $G\in L^2(\Omega,\mthree)$
such that $G^h\wto G$ weakly in $L^2(\Omega,\mthree)$.
Moreover, by Lemma~\ref{symmpartG} (see the end of this section) the symmetric part of $G$ can be
characterized as follows: 
there exists $\beta \in L^2(\Omega,\R^3)$, with zero average on $S$
and $\partial_k\beta\in L^2(\Omega,\R^3)$ for $k=2,3$, such that, if we set 
$$
M(\beta):=\Big( x_2A'e_2+ x_3A'e_3\,\Big| \,\partial_2 \beta 
\, \Big| \, \partial_3 \beta\Big),
$$
we have
\begin{equation}\label{carG}
\sym\, G= 
\begin{cases}
\sym\, M(\beta)+(u'+\frac{1}{2}[(v_2')^2+(v_3')^2])e_1 \otimes e_1 & 
\text{if }\alpha=3,
\\
\sym\, M(\beta)+u' e_1 \otimes e_1 & 
\text{if }\alpha>3,
\\
\sym\, M(\beta)+g e_1 \otimes e_1 & 
\text{if }2<\alpha<3
\end{cases}
\end{equation}
for some $g \in L^2(0,L)$. In particular, by
the normalization hypotheses (\ref{ipsimm}) on $S$ we deduce
\begin{equation}\label{barG11}
\lin{{G}}_{11}=
\begin{cases}
u'+\frac{1}{2}[(v_2')^2+(v_3')^2]&\textrm{ for }\alpha=3,\\
u'&\textrm{ for }\alpha>3,\\
g &\textrm{ for }2<\alpha<3.
\end{cases}
\end{equation}
\medskip

\noindent
\textbf{Step 2.} \textit{Stress tensor estimate}
\smallskip

\noindent
We define the stress $E^h:\Omega\to\mthree$ as
\begin{equation}\label{defeh}
E^h=\frac{1}{h^{\alpha-1}}DW(Id+h^{\alpha-1}G^h)(Id+h^{\alpha-1}G^h)^T.
\end{equation}
{}From the frame indifference of $W$ it follows that
$$
DW(F)F^T=F(DW(F))^T \quad\text{for every }
F\in\mthree_+.
$$
This implies that $E^h$ is symmetric for every $h>0$. Moreover, the following
pointwise estimate holds:
\begin{equation}\label{stimaeh}
|E^h|\leq C \Big(\frac{W(Id+h^{\alpha-1}G^h)}{h^{\alpha-1}}+|G^h|\Big).
\end{equation}
Indeed, let $\delta$ be the width of the neighbourhood of $SO(3)$ where $W$ is
of class $C^2$. Suppose first that $h^{\alpha-1}|G^h|\leq \frac{\delta}{2}$. 
Then, a first order Taylor expansion of $DW$ around the identity, together with (H4) and (H5),
yields
$$
DW(Id+h^{\alpha-1}G^h)=h^{\alpha-1}D^2W(M^h)G^h
$$ 
for some $M^h\in\mthree$ satisfying $|M^h-Id|\leq \frac{\delta}{2}$.
Since $D^2W$ is bounded on the set $\{F\in\mthree:
\dist(F,SO(3))\leq\frac{\delta}{2}\}$, 
we deduce 
$$
|DW(Id+h^{\alpha-1}G^h)|\leq Ch^{\alpha-1}|G^h|.
$$
Therefore, by (\ref{defeh}) we obtain
$$
|E^h|\leq C|G^h|+Ch^{\alpha-1}|G^h|^2\leq C(1+\delta)|G^h|.
$$
If instead $h^{\alpha-1}|G^h|>\frac{\delta}{2}$, we first observe that 
$W(\nablah y^h)$ is finite a.e.\ in $\Omega$ by \eqref{energiapiccola2}. 
By (H2) and by frame indifference we deduce
$$
\det(\nablah y^h)=\det(Id+h^{\alpha-1}G^h)>0 \textrm{ a.e.\ in }\Omega.
$$
Therefore, we can use (H7), which yields 
$$
|E^h|\leq \frac{1}{h^{\alpha-1}}k(W(Id+h^{\alpha-1}G^h)+1)\leq k
\frac{W(Id+h^{\alpha-1}G^h)}{h^{\alpha-1}}+\frac{2k}{\delta}|G^h|.
$$
This completes the proof of (\ref{stimaeh}).
\medskip

\noindent
\textbf{Step 3.} \textit{Convergence properties of the scaled stress}
\smallskip

\noindent
Arguing as in \cite{M-S}, 
some convergence properties of the stresses $E^h$ can be deduced from (\ref{stimaeh}). 
Indeed, using \eqref{energiapiccola2} and the fact that the $G^h$ are bounded
in $L^2(\Omega,\mthree)$, we obtain  
from (\ref{stimaeh}) that for each measurable set $\Lambda$ the following
estimate holds true:
\begin{equation}\label{inslambda}
\int_{\Lambda}{|E^h|dx}\leq Ch^{\alpha-1}+C|\Lambda|^{\frac{1}{2}},
\end{equation}
where $|\Lambda|$ denotes the Lebesgue measure of $\Lambda$.
Let now 
\begin{equation}\label{defbh}
B_h:=\{x \in \Omega: \ h^{\alpha-1-\gamma}|G^h(x)|\leq 1\},
\end{equation}
where $\gamma \in (0,\alpha-2)$, and let 
$\chi_h$ be the characteristic function of $B_h$. By \eqref{inslambda} and
by Chebyshev inequality we have
\begin{equation}\label{omegamenbheh}
\int_{\Omega \setminus B_h}{|E^h|dx}\leq Ch^{\alpha-1-\gamma},
\end{equation}
so that 
\begin{equation}\label{1menchih}
(1-\chi_h)E^h\to0 \quad\text{strongly in }L^1(\Omega,\mthree).
\end{equation}
Moreover, one can show that the remainder in the first order Taylor expansion of
$DW(Id+h^{\alpha-1}G^h)$ around the identity is uniformly controlled on the sets
$B^h$, so that 
\begin{equation}\label{chiheh}
\chi_hE^h \deb \cal LG=:E \quad\text{in }L^2(\Omega,\mthree)
\end{equation}
(see Step~3 in the proof of \cite[Theorem~3.1]{M-S} for details).
\medskip

\noindent
\textbf{Step 4.} \textit{Consequences of the Euler-Lagrange equations}
\smallskip

\noindent
By the frame indifference of $W$ and by \eqref{dec} we have
$$
DW(\nablah y^h)(\nablah y^h)^T=h^{\alpha-1}R^hE^h(R^h)^T.
$$
Therefore, the Euler-Lagrange equations \eqref{2pstcresc} can be written as 
\begin{equation}\label{psstcresc}
\myintom{R^hE^h(R^h)^T{\,:\,}[(\nabla \phi)\comp y^h]}
=h\myintom{[f_2(\phi_2\comp y^h)+f_3(\phi_3\comp y^h)]}
\end{equation}
for every $\phi \in C^1_b(\R^3,\R^3)$ satisfying the boundary
condition $\phi(0,hx_2,hx_3)=0$ for all $(x_2,x_3)\in S$.

Let now $\phi$ be a function in $C^1_b(\R^3,\R^3)$ such that 
$\phi(0,x_2,x_3)=0$ for every $(x_2,x_3)\in S$. For each $h>0$ we define
$$
\phi^h(x):=h\phi\Big(x_1,\frac{x_2}{h}-\frac{1}{h}\xi^h_2(x_1),\frac{x_3}{h}-\frac{1}{h}\xi^h_3(x_1)\Big),
$$
where $\xi^h_2$, $\xi^h_3$ are the functions constructed in Lemma~\ref{succappr}.  
By (\ref{bordoxih}) the maps $\phi^h$ are admissible test functions in (\ref{psstcresc}).

To simplify computations we introduce the following notation: 
\begin{equation}
\label{szh}
z^h:=\Big(y^h_1,\frac{y^h_2}{h}-\xi^h_2\comp y^h_1,\frac{y^h_3}{h}-\frac{1}{h}\xi^h_3\comp y^h_1\Big).
\end{equation}
{}From \eqref{yhconv} and (\ref{convxih}) it follows that
\begin{equation}\label{convi}
z^h\to x \quad \textrm{ in }L^2(\Omega,\R^3).
\end{equation}
Choosing $\phi^h$ as test function in (\ref{psstcresc}) we obtain
\begin{multline}\label{sost1}
\myintom{R^hE^h(R^h)^Te_1 {\,\cdot\,} \Big[h\partial_1 \phi\comp z^h
-\sum_{k=2}^3{(\partial_k\phi\comp z^h) \big((\xi^h_k)'\comp y^h_1\big)}\Big]}
\\ 
{}+\myintom{\sum_{k=2}^{3} R^hE^h(R^h)^Te_k {\,\cdot\,} (\partial_k \phi\comp z^h)}+\myintom{h^2\big[f_2(\phi_2\comp z^h)+f_3(\phi_3\comp z^h)\big]}=0.
\end{multline}
By \eqref{inslambda} and \eqref{limitazionenorme} we have
\begin{multline}\nonumber
\bigg|\myintom{R^hE^h(R^h)^Te_1 {\,\cdot\,} \Big[h\partial_1 \phi\comp z^h-
\sum_{k=2}^3{(\partial_k\phi\comp z^h)\big((\xi^h_k)'\comp y^h_1\big)}\Big]}\bigg|
\\
\leq C\|E^h\|_{L^1}\Big(
\|h\partial_1\phi\|_{L^{\infty}}+\sum_{k=2}^3\|\partial_k\phi\|_{L^{\infty}}\|(\xi^h_k)'\|_{L^{\infty}}
\Big)\leq C(h+h^{\alpha-2}),
\end{multline}
therefore the first integral in \eqref{sost1} converges to zero. Analogously, 
since $f_k\in L^2(0,L)$ for $k=2,3$ and $\phi_k\in C^1_b(\R)$, 
the last integral in \eqref{sost1} tends to zero. 
We deduce that the second integral in (\ref{sost1}) must also converge to zero. 
On the other hand, this term can be written as 
\begin{eqnarray}
\nonumber 
\lefteqn{\myintom{\sum_{k=2}^{3}{R^hE^h(R^h)^Te_k {\,\cdot\,} (\partial_k
\phi\comp z^h)}}}
\\
& = &  \myintom{\sum_{k=2}^{3}{\chi_hR^hE^h(R^h)^Te_k {\,\cdot\,} (\partial_k
\phi\comp z^h)}} \nonumber
\\
& &{}+\myintom{\sum_{k=2}^{3}{(1-\chi_h)R^hE^h(R^h)^Te_k {\,\cdot\,}
(\partial_k \phi\comp z^h)}}.
\label{3terms} 
\end{eqnarray}
By (\ref{convi}) and by the dominated convergence theorem we have 
\begin{equation}\label{convcosevarie}
\partial_k \phi\comp z^h\to\partial_k\phi\quad\textrm{ in }L^2(\Omega).
\end{equation}
Thus, by (\ref{chiheh}) and  by the fact that $R^h\to Id$ in $L^{\infty}(0,L)$
we deduce
$$ 
\myintom{\sum_{k=2}^{3}{\chi_hR^hE^h(R^h)^Te_k {\,\cdot\,} (\partial_k
\phi\comp z^h)}}\to
 \displaystyle\displaystyle\myintom{\sum_{k=2}^{3}Ee_k{\,\cdot\,}\partial_k
\phi},
$$
while by (\ref{1menchih}) we have that the last term in \eqref{3terms} tends to zero.
We conclude that
\begin{equation}\label{seconda}
\myintom{\sum_{k=2}^{3}Ee_k{\,\cdot\,}\partial_k \phi}=0 
\end{equation}
for every $\phi \in C^1_b(\R^3,\R^3)$ such that $\phi(0,x_2,x_3)=0$ 
for all $(x_2,x_3)\in S$. Therefore, the following equations hold true a.e.\ in $(0,L)$:
\begin{equation}\label{div}
\begin{cases}
\div_{x_2,x_3}(Ee_2|Ee_3)=0 & \text{ in } S,
\smallskip \\
(Ee_2|Ee_3)\nu_{\partial S}=0 & \text{ on }\partial S,
\end{cases}
\end{equation}
where $\nu_{\partial S}$ is the unit normal to $\partial S$. Moreover,
for a.e.\ $x_1\in (0,L)$
\begin{equation}\label{eek}
\myintegra{Ee_k}=0\quad \textrm{ for } k=2,3. 
\end{equation}
We conclude that $\lin{E}e_2=\lin{E}e_3=0$ a.e.\ in $(0,L)$ and
since $E$ is symmetric, 
$$
\lin{E}=\lin{E}_{11}e_1 \otimes e_1.
$$
\medskip

\noindent
\textbf{Step 5.} \textit{Zeroth moment of the Euler-Lagrange equations}
\smallskip

\noindent
We now identify the zeroth order moment of the limit stress $E$. Let
$\psi$ be a function
in $C^1_b(\R)$ such that $\psi(0)=0$. We define
$$
\phi(x)=\psi(x_1)e_1.
$$
Using $\phi$ as a test function in the Euler-Lagrange equations
\eqref{psstcresc} we have
\begin{equation}\label{serveallafine}
\myintom{(R^h E^h (R^h)^T)_{11}(\psi'\comp y_1^h)}=0.
\end{equation}
To pass to the limit in the previous equation, we split $\Omega$ into the sets 
$B_h$ and $\Omega\setminus B_h$, so that we obtain
\begin{equation}\label{sommaint}
\myintom{\chi_h(R^hE^h(R^h)^T)_{11}(\psi'\comp y_1^h)}+\myintom{
(1-\chi_h)(R^hE^h(R^h)^T)_{11}(\psi'\comp y_1^h)}=0.
\end{equation}
By \eqref{yhconv} and by the continuity of $\psi'$ it follows that 
$\psi'\comp y_1^h$ converges to $\psi'$ in $L^2(\Omega)$. Therefore,
by (\ref{1menchih}) and (\ref{chiheh}) we can pass to the limit in (\ref{sommaint})
and we deduce
$$
\myintol{\lin{E}_{11}\psi'}=\myintom{E_{11}\psi'}=0
$$
for every $\psi \in C^1_b(\R)$ such that $\psi(0)=0$. 
This implies that $\lin E=\lin{E}_{11}e_1\otimes e_1=0$ a.e.\ in $(0,L)$.

Since by frame indifference $\cal L H=0$ for every skew-symmetric $H\in\mthree$, we obtain that $\cal L\,\lin{\sym\, G}=\cal L
\lin{G}=\lin{E}=0$. The invertibility of $\cal L$ on the space of symmetric matrices
yields that $\lin{\sym\, G}=0$. 
Together with \eqref{barG11}, this implies
(\ref{eq1a}) for $\alpha=3$, (\ref{eq1b}) for $\alpha>3$,
and $g=0$ a.e.\ in $(0,L)$ for $2<\alpha<3$.
Moreover, by \eqref{carG} we deduce that
$$
\sym\Big( 0 \, \Big| \int_S\partial_2 \beta\, dx_2dx_3\, \Big| \, \int_S
\partial_3 \beta\, dx_2dx_3\Big)=0,
$$
so that, if we introduce $\tilde\beta:\Omega\to\R^3$ defined by
$$
\tilde\beta:=\Big(\beta_1, \beta_2-x_3\int_S\partial_3\beta_2\,dx_2dx_3, 
\beta_3-x_2\int_S\partial_2\beta_3\,dx_2dx_3\big),
$$
we have that  $\tilde \beta(x_1,\cdot,\cdot)\in \cal B$ for a.e.\ $x_1\in (0,L)$ and
$$
\sym\, G= \sym\Big( x_2A'e_2 + x_3A'e_3\, \Big|
\, \partial_2 \tilde\beta \, \Big| \, \partial_3 \tilde\beta\Big).
$$
In particular, we have the following characterization of $E$:
$$
E=\cal L\,\sym\, G= \cal L\Big( x_2A'e_2 + x_3A'e_3\, \Big|
\, \partial_2 \tilde\beta \, \Big| \, \partial_3 \tilde\beta\Big).
$$ 
Since $E$ satisfies \eqref{div}, we deduce from Lemma~\ref{lem3} that
$\tilde \beta$ is a minimizer of the functional 
$$
\cal G_{A'}(\beta)=\myintegra{Q_3\Big( x_2A'e_2 + x_3A'e_3\, \Big|
\, \partial_2 \beta \, \Big| \, \partial_3 \beta\Big)}.
$$
In other words, $\tilde \beta$ satisfies 
\begin{equation}
Q_1(A')=\myintegra{Q_3\Big( x_2A'e_2 + x_3A'e_3\, \Big|
\, \partial_2 \tilde\beta \, \Big| \, \partial_3 \tilde\beta\Big)}
\end{equation}
for all $\alpha>2$.
\medskip

\noindent
\textbf{Step 6.} \textit{First-order moments of the Euler-Lagrange equations}
\smallskip

\noindent
In this step we prove that the limiting Euler-Lagrange equations (\ref{eq2a})
and (\ref{eq2b}) are satisfied. Let $\varphi_2,\varphi_3$ be
two functions in $C^1_b(\R)$ with $\varphi_2(0)=\varphi_3(0)=0$. We
define
$$
\phi^h(x)=\Big(0,\frac{\varphi_2(x_1)}{h},\frac{\varphi_3(x_1)}{h}\Big)
$$
and we use $\phi^h$ as test function in (\ref{psstcresc}).
By \eqref{yhconv} the force term can be treated as follows:
\begin{eqnarray}
 \nonumber\lim_{h\to0}{\myintom{h[f_2 (\phi_2^h\comp y^h)+f_3
(\phi_3^h\comp y^h)]}}&\hspace{-0.2 cm}=\hspace{-0.2 cm}&
\lim_{h\to0}{\myintom{[f_2 (\varphi_2 \comp y^h_1)+f_3 (\varphi_3\comp
y_1^h)]}}\\
&\hspace{-0.2 cm}=\hspace{-0.2 cm}&\myintol{(f_2 \varphi_2+f_3\varphi_3)}.
\end{eqnarray}
Therefore, we have
\begin{eqnarray}
\nonumber &
\displaystyle\lim_{h\to0}{\myintom{\Big[(R^hE^h(R^h)^T)_{21}\frac{
\varphi_2'\comp y^h_1}{h}
+(R^hE^h(R^h)^T)_{31}\frac{\varphi_3'\comp y^h_1}{h}\Big]}}&\\
\label{limboh} &\displaystyle=\myintol{(f_2
\varphi_2+f_3 \varphi_3)}.&
\end{eqnarray}
We shall characterize the limit on the left-handside of \eqref{limboh} in terms of the
first-order moments of the stress $E$. To this aim, we go back to the Euler-Lagrange 
equations (\ref{psstcresc}) and we construct some ad-hoc
test functions with a linear behaviour in the variables $x_2$, $x_3$. 
Let $(\omega_h)$ be a sequence of positive numbers such that
\begin{eqnarray}
\label{oh1}
& & h \omega_h\to +\infty,\\
\label{oh2}
& & h^{\alpha-1-\gamma}\omega_h\to0,
\end{eqnarray}
where $\gamma \in (0,\alpha-2)$ is the same exponent introduced in
(\ref{defbh}). 
{}For each $h>0$ we consider a function $\theta^h \in C^1_b(\R)$ which
coincides with the identity in a large enough neighbourhood of the origin, that is,
\begin{equation}\label{th1}
\theta^h(t)=t\quad \textrm{ for }|t|\leq \omega_h
\end{equation}
and, in addition, satisfies the following properties:
\begin{eqnarray}
\label{th2}
&|\theta^h(t)|&\leq |t|\quad \forall t \in\R,\\
\label{th3}
&\|\theta^h\|_{L^{\infty}}&\leq 2 \omega_h,\\
\label{th4}
&\Big\|\dfrac{d \theta^h}{dt}\Big\|_{L^{\infty}}\hspace{-0.1 cm}&\leq 2.
\end{eqnarray}
Let $\eta$ be a function in $C^1(\R)$ with compact support and such that $\eta(0)=0$, 
and let $\xi^h_k$, $k=2,3$, be the functions constructed in Lemma~\ref{succappr}.
We consider the map
$$
\phi^h(x)=\theta^h\Big(\frac{x_3}{h}-\frac{1}{h}\xi^h_3(x_1)\Big)\eta(x_1)e_1.
$$
Choosing $\phi^h$ as test function in (\ref{psstcresc}) and using the
notation introduced in \eqref{szh}, we obtain
\begin{eqnarray}
\nonumber
&& \myintom{(R^hE^h(R^h)^T)_{11}(\theta^h \comp z^h_3)(\eta'\comp
y^h_1)}\\
\nonumber &\hspace{-0.2 cm}-\hspace{-0.2 cm}&\myintom
{\frac { (R^hE^h(R^h)^T)_{11} } {
h}\Big(\frac{d\theta^h}{dt}\comp z^h_3\Big)[(\xi^h_3)'\comp y^h_1](\eta\comp
y^h_1)}\\
\label{sost2} &\hspace{-0.2
cm}+\hspace{-0.2 cm}&\myintom{(R^hE^h(R^h)^T)_{13}\frac{\eta\comp
y^h_1}{h}\Big(\frac{d\theta^h}{dt}
\comp z^h_3\Big)}=0.
\end{eqnarray}
The first integral in \eqref{sost2} can be decomposed into the sum of two terms
\begin{eqnarray}
\nonumber &\displaystyle
\myintom{(R^hE^h(R^h)^T)_{11}(\theta^h\comp
z^h_3)(\eta'\comp y^h_1)}&\\
\nonumber &\hspace{-1.1 cm}=&\hspace{-3.6 cm}\myintom {
\chi_h[(R^hE^h(R^h)^T)_{11}
(\theta^h\comp z^h_3)(\eta'\comp y^h_1)]}\\ 
 \label{ssost2}
&\hspace{-1.1 cm}+&\hspace{-3.6 cm}\myintom{(1-\chi_h)\Big[(R^hE^h(R^h)^T)_{11}
(\theta^h\comp z^h_3)(\eta'\comp y^h_1)\Big]}.
\end{eqnarray}
By \eqref{yhconv}, \eqref{convi}, (\ref{th2}), and by the dominated
convergence theorem we deduce that
$$
(\theta^h\comp z^h_3)(\eta'\comp y^h_1)\to x_3\eta' \textrm{ in
}L^2(\Omega).
$$
Therefore, by \eqref{chiheh} we have
$$ 
\lim_{h\to0}{\myintom{\chi_h(R^hE^h(R^h)^T)_{11}(\eta'\comp y^h_1)(\theta^h \comp
z^h_3)}}
=\myintom{x_3 E_{11}\eta'}=\myintol{\capp{E}_{11}\eta'}.
$$
The second term in \eqref{ssost2} can be estimated using (\ref{omegamenbheh}),
as follows:
\begin{eqnarray}
 \nonumber\myintom{(1-\chi_h)|(R^hE^h(R^h)^T)_{11}
(\eta'\comp y^h_1)(\theta^h\comp z^h_3)| }
&\hspace{-0.2 cm}\leq\hspace{-0.2 cm}& 2
\omega_h\|\eta'\|_{L^\infty(0,L)}\int_{\Omega \setminus
B_h}{|E^h|}\\
\nonumber &\hspace{-0.2 cm}\leq\hspace{-0.2 cm}&\ Ch^{\alpha-1-\gamma}\omega_h,
\end{eqnarray}
and this latter is infinitesimal owing to (\ref{oh2}).
We conclude that
\begin{equation}\label{concl}
 \myintom{(R^hE^h(R^h)^T)_{11}(\theta^h\comp z^h_3)(\eta'\comp y^h_1)} \to
\myintol{\capp{E}_{11}\eta'}.
\end{equation} 
As for the second integral in \eqref{sost2}, we consider the following
decomposition:
\begin{eqnarray}
\nonumber
& \displaystyle\myintom
{(R^hE^h(R^h)^T)_{11} \frac{1}{h}\Big(\frac{d\theta^h}{dt}\comp
z^h_3\Big)[(\xi^h_3)'\comp y^h_1](\eta \comp y^h_1)}& \\
&\hspace{-4 cm}=&\hspace{-6 cm}\myintom {
(R^hE^h(R^h)^T)_{11}\Big[\Big(\frac{d\theta^h}{dt}
\comp z^h_3\Big)-1\Big]\frac{1}{h}[(\xi^h_3)'\comp y^h_1](\eta\comp y^h_1)}
\nonumber
\\
\label{dec2}&\hspace{-4 cm}+&\hspace{-6
cm}\myintom{(R^hE^h(R^h)^T)_{11}\frac{1}{h}[ (\xi^h_3)'\comp y^h_1](\eta
\comp y^h_1) } .
\end{eqnarray}
To study the first term in \eqref{dec2} we introduce the sets 
\begin{equation}\label{tilDh}
D_h=\{x \in \Omega : |z^h_3(x)|\geq \omega_h\}.
\end{equation}
Since $(z^h_3)$ is uniformly bounded in $L^2(\Omega)$, by Chebyshev
inequality we deduce that
\begin{equation}\label{normaDh}
|D_h|\leq C \omega_h^{-2}.
\end{equation}
Thus, by (\ref{limitazionenorme}) and (\ref{inslambda}) we have
\begin{align}
\Big|\int_\Omega &(R^hE^h(R^h)^T)_{11}\Big[\Big(\frac{d\theta^h}{dt}\comp
z^h_3\Big)-1\Big]\frac{1}{h}[(\xi^h_3)'\comp y^h_1](\eta\comp y^h_1)\, dx
\Big|
\nonumber
\\
& \leq \int_{D_h}{\Big|(R^hE^h(R^h)^T)_{11}\Big[\Big(\frac{d\theta^h
}{dt}\comp z^h_3\Big)-1\Big]\frac{1}{h}[(\xi^h_3)'\comp
y^h_1](\eta\comp y^h_1)\Big|}
\nonumber
\\
&\label{thus} \leq  Ch^{\alpha-3}\int_{D_h}{|E^h|}\leq
Ch^{\alpha-3}(h^{\alpha-1}+|D_h|^{\frac{1}{2}})\leq
C\Big(h^{2\alpha-4}+\frac{h^{\alpha-2}}{h\omega_h}\Big),
\end{align}
where the latter term tends to zero owing to (\ref{oh1}). Furthermore, we can
prove that the second term in \eqref{dec2} is equal to
zero. Indeed, let
$$
\psi^h(x_1):=\int_0^{x_1}{\frac{1}{h}(\xi^h_3)'(s)\eta(s)ds}.
$$
It is easy to verify that $\psi^h\in C^1_b(\R)$ and $\psi^h(0)=0$ for
every $h>0$. Therefore, by \eqref{serveallafine} we obtain 
$$
\myintom{(R^hE^h(R^h)^T)_{11}\frac{1}{h} [(\xi^h_3)'\comp y^h_1](\eta
\comp y^h_1) }=0.
$$
By \eqref{dec2} and \eqref{thus} we conclude that
\begin{equation}\label{ten0}
\myintom{ (R^hE^h(R^h)^T)_{11}\frac{1}{h}\Big(\frac{d\theta^h}{dt}\comp
z^h_3\Big)[(\xi^h_3)'\comp y^h_1](\eta \comp y^h_1)}\to0.
\end{equation} 
It remains to study the third integral in \eqref{sost2}, which can be
written as
\begin{eqnarray}
 &\displaystyle\nonumber\myintom{(R^hE^h(R^h)^T)_{13}\frac{\eta \comp
y^h_1}{h}\Big(\frac{
d\theta^h } { dt }
\comp z^h_3 \Big)}&\\
\nonumber&\hspace{-3.6
cm}=&\hspace{-4.8 cm}\myintom{(R^hE^h(R^h)^T)_{13}\frac{\eta \comp y^h_1}{h}
\Big[\Big(\frac { d\theta^h}{dt}\comp z^h_3\Big)-1\Big]}\\
 \label{nevap}&\hspace{-3.6
cm}+&\hspace{-4.8 cm}\myintom{(R^hE^h(R^h)^T)_{13}\frac{\eta \comp y^h_1}{h}}.
\end{eqnarray}
We claim that
\begin{equation}\label{altraclaim}
\lim_{h\to0}{\myintom{(R^hE^h(R^h)^T)_{13}\frac{\eta
\comp y^h_1}{h}\Big[\Big(\frac{d\theta^h}{dt}
\comp z^h_3 \Big)-1\Big]}}=0.
\end{equation}
To prove it, we consider again the sets $D_h$ defined in (\ref{tilDh}). 
{}From (\ref{inslambda}), (\ref{th4}), (\ref{normaDh}),
and from the boundedness of $\eta$ we obtain
\begin{eqnarray*}
\lefteqn{\myintom{\Big|(R^hE^h(R^h)^T)_{13}\frac{\eta \comp
y^h_1}{h}\Big[\Big(\frac{ d\theta^h
}{dt}\comp z^h_3\Big)-1\Big]\Big|}}
\\
& = &\int_{D_h}{\Big|(R^hE^h(R^h)^T)_{13}\frac{\eta \comp
y^h_1}{h}\Big[\Big(\frac{
d\theta^h}{dt}\comp z^h_3\Big)-1\Big]\Big|}
\\
& \leq & \frac{C}{h}\int_{D_h}{|E^h|}\leq \frac{C}{h}(h^{\alpha-1}+|D_h|^{\frac{1}{2}})\leq C\Big(h^{\alpha-2}+\frac{1}{h\omega_h}\Big)
\end{eqnarray*}
and the latter is infinitesimal owing to (\ref{oh1}), so that
(\ref{altraclaim}) follows.
In conclusion, combining (\ref{sost2}), \eqref{concl}, and \eqref{ten0}--\eqref{altraclaim} we
deduce that
\begin{equation}\label{sost2meglio}
\lim_{h\to0}\myintom{(R^hE^h(R^h)^T)_{13}\frac{\eta
\comp y^h_1}{h}}=-\myintom{\capp{E}_{11} \eta'}
\end{equation}
for every $\eta \in C^1(\R)$ with compact support and such that $\eta(0)=0$. Choosing a test
function of the form 
$$
\phi^h(x)=\theta^h\Big(\frac{x_2}{h}-\frac{1}{h}
\xi^h_2(x_1)\Big)\eta(x_1)e_1,
$$
one can prove analogously that
\begin{equation}\label{sost2megliob}
 \lim_{h\to0}\myintom{(R^hE^h(R^h)^T)_{12}\frac{\eta
\comp y^h_1}{h}}=-\myintom{\til{E}_{11} \eta'}.
\end{equation}
Let now $\varphi_k \in C^2(\R)$ with compact support be such that
$\varphi_k(0)=\varphi_k'(0)=0$
for $k=2,3$. We choose $\eta=\varphi'_3$ in \eqref{sost2meglio} and
$\eta=\varphi'_2$ in \eqref{sost2megliob} and we add the two equations. 
Comparing with \eqref{limboh} and using the fact that 
$E^h$ (and therefore, $R^hE^h(R^h)^T$) is symmetric, we conclude that
$$
\myintol{\til{E}_{11}\varphi_2''+\capp{E}_{11}
\varphi_3''+f_2\varphi_2+f_3\varphi_3}=0
$$
for every $\varphi_k \in C^2(\R)$  with compact support and such that
$\varphi_k(0)=\varphi_k'(0)=0,$ $k=2,3$. By approximation we obtain
(\ref{eq2a}) and (\ref{eq2b}) for all $\alpha>2$.
\medskip

\noindent
\textbf{Step 7.} \textit{Euler-Lagrange equation for the twist function}
\smallskip

\noindent
To conclude the proof of the theorem, it remains to verify the 
limiting Euler-Lagrange equation \eqref{eq3}. 
We define
$$
\phi^h(x)=\Big(0,-\theta^h\Big(\frac{x_3}{h}-\frac{\xi^h_3(x_1)}{h}\Big)\eta(x_1),\theta^h\Big(\frac{x_2}{h}-\frac{\xi^h_2(x_1)}{h}\Big)\eta(x_1)\Big),
$$
where $\eta \in C^1(\R)$ with compact support, $\eta(0)=0$, and $\theta^h$ is as in Step~6.
Using $\phi^h$ as test function in the Euler-Lagrange equations
(\ref{psstcresc}), we obtain
\begin{eqnarray}
\lefteqn{{}- \myintom{\Big[(R^hE^h(R^h)^T)_{21}(\theta^h\comp z^h_3)-
(R^hE^h(R^h)^T)_{31}(\theta^h\comp z^h_2)\Big]
(\eta'\comp y_1^h)
}}\nonumber
\\
& &
{}+\myintom{(R^hE^h(R^h)^T)_{21}\Big(\frac{d\theta^h}{dt}\comp z^h_3\Big)((\xi^h_3)'\comp y^h_1)\frac{\eta\comp y^h_1}{h}}\nonumber
\\
& &{}-\myintom{(R^hE^h(R^h)^T)_{31}\Big(\frac{d\theta^h}{dt}\comp z^h_2\Big)((\xi^h_2)'\comp y^h_1) 
\frac{\eta\comp y^h_1}{h}
}\nonumber
\\
& &{}+\myintom{\Big[(R^hE^h(R^h)^T)_{32}\Big(\frac{d\theta^h}{dt}\comp
z^h_2\Big)
-(R^hE^h(R^h)^T)_{23}\Big(\frac{d\theta^h}{dt}\comp
z^h_3\Big)\Big]\frac{\eta\comp y^h_1}{h}}\nonumber
\\
& & {}+h \myintom{\big[f_2(\theta^h\comp z^h_3)-f_3
(\theta^h\comp z^h_2)\big](\eta\comp y^h_1)}=0. \label{pstcresc3}
\end{eqnarray}
Arguing as in the proof of \eqref{concl}, we can show that the first integral
in \eqref{pstcresc3} satisfies
\begin{eqnarray}
\nonumber&\displaystyle\lim_{h\to0}{\displaystyle\myintom{\Big[(R^hE^h(R^h)^T)_{21}(\theta^h\comp z^h_3)-
(R^hE^h(R^h)^T)_{31}(\theta^h\comp z^h_2)\Big]
(\eta'\comp y_1^h)
}}&
\\\nonumber&=&\hspace{-5.8 cm}\myintol{(-\capp{E}_{12}+ \til{E}_{13})\eta'}.
\end{eqnarray}
The proof of \eqref{eq3} is concluded if we show that all other terms
in \eqref{pstcresc3} converge to zero, as $h\to0$. 
The last integral in \eqref{pstcresc3} is infinitesimal owing to the estimate 
\begin{eqnarray}
\nonumber\Big|h \myintom{[f_2
(\theta^h \comp z^h_3)-f_3 (\theta^h \comp z^h_2)](\eta \comp y^h_1)}\Big|&\hspace{-0.2
cm}\leq&\hspace{-0.2 cm}
Ch(\|f_2\|_{L^2}\|z^h_3\|_{L^2}+\|f_3\|_{L^2}\|z^h_2\|_{L^2})\\
\nonumber &\hspace{-0.2 cm}
\leq&\hspace{-0.2 cm}
Ch,
\end{eqnarray}
which follows from (\ref{th2}) and (\ref{convi}).

As for the term
$$
\myintom{\Big[(R^hE^h(R^h)^T)_{32}\Big(\frac{d\theta^h}{dt}\comp
z^h_2\Big)
-(R^hE^h(R^h)^T)_{23}\Big(\frac{d\theta^h}{dt}\comp
z^h_3\Big)\Big]\frac{\eta\comp y^h_1}{h}},
$$ 
we remark that by the
symmetry of $R^hE^h(R^h)^T$ it can be written as
$$
\myintom{\frac{\eta \comp y^h_1}{h}(R^hE^h(R^h)^T)_{32}\Big\{
\Big[\Big(\frac{d\theta^h}{dt}\comp
z^h_2\Big)-1\Big]+\Big[1-\Big(\frac{d\theta^h}{dt}
\comp z^h_3\Big)\Big]\Big\}}.
$$ 
Arguing as in the proof of (\ref{altraclaim}), we obtain
that the above expression tends to zero, as $h\to0$.

It remains to prove that
$$
\lim_{h\to0}\myintom{\frac{1}{h}(R^hE^h(R^h)^T)_{k1}\Big(\frac{d\theta^h}{dt}
\comp z^h_j\Big)[(\xi^h_j)'\comp y^h_1](\eta \comp y^h_1)}=0
$$
for $k,j\in\{2,3\}$, $k\neq j$.
To this aim, we fix $k=2,j=3$ and we write the previous integral as the
sum of two terms
\begin{eqnarray}
\nonumber  & &\myintom{\frac{1}{h}(R^hE^h(R^h)^T)_{21}\Big(\frac{d\theta^h}{dt}
\comp z^h_3\Big)[(\xi^h_3)'\comp y^h_1](\eta \comp y^h_1)}\\
&
&\nonumber=\myintom{\frac{1}{h}(R^hE^h(R^h)^T)_{21}\Big[\Big(\frac{d\theta^h}{
dt}\comp z^h_3\Big)-1\Big][(\xi^h_3)'\comp y^h_1](\eta \comp y^h_1)}\\
& &\label{2termsb}+\myintom{\frac{
(R^hE^h(R^h)^T)_ {21}}{h^{1-\epsilon}}\frac{
(\xi^h_3)'\comp y^h_1}{h^{\epsilon}}(\eta \comp y^h_1)}.
\end{eqnarray}
where $0<\epsilon<\alpha-2$.
Arguing as in the proof of (\ref{altraclaim}), we obtain that the
first term is infinitesimal. To study the second term, we notice that, if 
$(\psi^h)\subset C^1_b(\R)$ is a sequence of functions such that $\psi^h(0)=0$ and
$\|\psi^h\|_{L^{\infty}(\R)}\leq C$ for all $h>0$, then the map $\psi^h(x_1)e_j$
can be used as test function in the Euler-Lagrange
equations (\ref{psstcresc}) for every $h>0$, and we have
\begin{equation}\label{last}
\Big|\myintom{\frac{(R^hE^h(R^h)^T)_{j1}}{h^{1-\epsilon}}[(\psi^h)'\comp
y^h_1]} \Big|=\Big|\myintom{h^{\epsilon}f_j(\psi^h \comp y^h_1)}\Big|\leq
Ch^{\epsilon}\|f_j\|_{L^2(\Omega)}\to0.
\end{equation}
If we now choose
$$
\psi^h(x_1):=\int_0^{x_1}{\frac{(\xi^h_k)'(s)}{h^{\epsilon}}\eta(s)ds},
$$
then by \eqref{limitazionenorme} we obtain
$$
\|\psi^h\|_{L^{\infty}}\leq Ch^{\alpha-2-\epsilon}\|\eta\|_{L^1}\leq C \textrm{
for all }h>0,
$$
so that by \eqref{last} also the last term in \eqref{pstcresc3}
is infinitesimal as $h\to0$. This concludes the proof of \eqref{eq3} and of
the theorem.
\end{proof}
 
We conclude the section with a lemma, which provides us with a characterization
of the limiting strain. This result is contained in the proof
of \cite[Theorems~4.3 and~4.4]{S}. We present here a concise proof for the reader's
convenience.

\begin{lem}\label{symmpartG}
Let all the assumptions of Theorem~\ref{teoa} be satisfied and let $(R^h)$ be
the sequence of rotations of Theorem~\ref{rot}. 
For every $h>0$ let $G^h:\Omega\to\mthree$ be defined by
$$
G^h=\frac{(R^h)^T\nablah y^h-Id}{h^{\alpha-1}}.
$$ 
and let $G$ be the weak limit of $(G^h)$ in $L^2(\Omega,\mthree)$ 
(which exists, up to subsequences, by (\ref{c2})).
Then, there exist $g\in L^2(0,L)$ and $\beta\in L^2(\Omega,\R^3)$, 
with zero average on $S$ and $\partial_k\beta\in L^2(\Omega,\R^3)$ for $k=2,3$, such that, if we define 
$$
M(\beta):=\Big(x_2A'e_2+x_3A'e_3\,\Big|\,
\partial_2 \beta \, \Big|\, \partial_3 \beta\Big),
$$
we have
\begin{equation}\label{charGG}
\sym\, G= 
\begin{cases}
\sym\, M(\beta)+(u'+\frac{1}{2}[(v_2')^2+(v_3')^2])e_1 \otimes e_1 & 
\text{if }\alpha=3,
\\
\sym\, M(\beta)+u' e_1 \otimes e_1 & 
\text{if }\alpha>3,
\\
\sym\, M(\beta)+g e_1 \otimes e_1 & 
\text{if }2<\alpha<3,
\end{cases}
\end{equation}
where $u$, $v_k$, and $A$ are the functions introduced in Theorem~\ref{teoa}.
\end{lem}

\begin{proof}
{}For every $h>0$ we consider the function $\gamma^h:\Omega\to\R^3$ defined
by
$$
\gamma^h(x):=\frac{1}{h^{\alpha}}[y^h(x)-hx_2R^h(x_1)e_2-hx_3R^h(x_1)e_3]
$$
for every $x\in\Omega$.
By \eqref{linfnorm} we have that 
\begin{equation}\label{derbetak}
\partial_k \gamma^h\deb Ge_k \quad \textrm{ for every }k=2,3. 
\end{equation}
Therefore, if we define $\beta^h:=\gamma^h-\lin\gamma^h$, where 
$\lin\gamma^h$ is the average of $\gamma^h$ on $S$, we deduce by Poincar\'e-Wirtinger inequality that $\beta^h$ is uniformly bounded in $L^2(\Omega, \R^3)$.
It follows that there exists $\tilde\beta\in L^2(\Omega, \R^3)$, with zero average on $S$, such that, up to subsequences, $\beta^h\deb \tilde \beta$ 
in $L^2(\Omega, \R^3)$. Furthermore, by \eqref{derbetak} we have
that $\partial_k \tilde \beta=Ge_k$ for all $k=2,3$.

As for the first column of $G$, we remark that by \eqref{ipsimm}
we can write
\begin{align}
\nonumber
R^hG^he_1&=h\partial_1\gamma^h+\frac{1}{h^{
\alpha-2}}(x_2(R^h)'e_2+x_3(R^h)'e_3) -\frac{1}{h^{\alpha-1}}R^he_1
\\
\label{rhghe1}
&= h\partial_1\beta^h+\frac{1}{h^{
\alpha-2}}(x_2(R^h)'e_2+x_3(R^h)'e_3)-\int_{S}{\frac{R^he_1-\partial_1 y^h}{h^{\alpha-1}}\, dx_2dx_3}.
\end{align}
By \eqref{linfnorm} we have that $R^hG^he_1\deb Ge_1$ weakly in
$L^2(\Omega,\mthree)$. Moreover, by \eqref{c2} there exists a function
$g\in L^2((0,L),\R^3)$ such that 
$$
\int_{S}{\frac{R^he_1-\partial_1 y^h}{h^{\alpha-1}}\,dx_2dx_3}\deb g \quad
\textrm{ weakly in }L^2((0,L), \R^3),
$$
while \eqref{ah} yields
$$
\frac{1}{h^{\alpha-2}}x_2(R^h)'e_2+x_3(R^h)'e_3
\deb x_2A'e_2+x_3A'e_3
\quad \textrm{ weakly in }L^2((0,L),\mthree).
$$
Finally, by the weak convergence of $(\beta^h)$ in $L^2(\Omega, \R^3)$
we have that $h\partial_1\beta^h\to0$ in $W^{-1,2}(\Omega, \R^3)$; thus,
passing to the limit in \eqref{rhghe1} we conclude that
$$
G=\Big(x_2A'e_2+x_3A'e_3+g \,\Big| \, \partial_2\tilde\beta
\, \Big| \, \partial_3\tilde\beta\Big).
$$
To obtain \eqref{charGG} it is now enough to define
$$
\beta:=\tilde \beta+x_2(g{\,\cdot\,} e_2)e_1+x_3(g{\,\cdot\,} e_3)e_1,
$$
so that
$$
\sym\,G=\sym\Big(x_2A'e_2+x_3A'e_3+(g{\,\cdot\,} e_1)e_1\,
\Big| \, \partial_2\beta \, \Big| \,\partial_3\beta\Big).
$$
This concludes the proof for $2<\alpha<3$.
For $\alpha\geq3$ a characterization of $g$ can be given. Indeed, one can observe that 
\begin{eqnarray*} 
\myintegra{\frac{\partial_1
y^h-R^he_1}{h^{\alpha-1}}{\,\cdot\,} e_1}&=&\myintegra{\frac{(\partial_1
y^h-1)+(1-R^h_{11})}{h^{\alpha-1}}}\\
&=&(u^h)'-h^{\alpha-3}\sym(R^h-Id)_{11},
\end{eqnarray*}
where $(u^h)$ is the sequence introduced in \eqref{defuh}. By \eqref{u} and
\eqref{syma} we obtain the thesis for $\alpha\geq 3$.
\end{proof}
\bigskip
\bigskip

\noindent
\textbf{Acknowledgements.}
This work is part of the project ``Metodi e modelli variazio\-nali in scienza dei materiali" 2010,
supported by GNAMPA.

\end{document}